\RequirePackage{fixltx2e}
\documentclass[english]{amsart}
\usepackage[T1]{fontenc}
\usepackage[latin9]{inputenc}
\usepackage{fancyhdr}
\usepackage{babel}
\usepackage{amstext}
\usepackage{amsthm}
\usepackage{amssymb}
\usepackage[all]{xy}
\usepackage[unicode=true,
 bookmarks=true,bookmarksnumbered=false,bookmarksopen=false,
 breaklinks=false,pdfborder={0 0 1},backref=false,colorlinks=false]
 {hyperref}
\hypersetup{pdftitle={Poisson Centralizer of the Trace},
 pdfauthor={Szabolcs M\'esz\'aros},
 pdfsubject={Poisson algebras}}

\makeatletter
\numberwithin{equation}{section}
\numberwithin{figure}{section}
\theoremstyle{plain}
\newtheorem{thm}{\protect\theoremname}[section]
  \theoremstyle{definition}
  \newtheorem{defn}[thm]{\protect\definitionname}
  \theoremstyle{remark}
  \newtheorem{rem}[thm]{\protect\remarkname}
  \theoremstyle{plain}
  \newtheorem{prop}[thm]{\protect\propositionname}

\newcommand{\noun}[1]{\textsc{#1}}
\usepackage{ifpdf}
\ifpdf
 \IfFileExists{lmodern.sty}{\usepackage{lmodern}}{}
\fi 
\usepackage[all]{xy}

\AtBeginDocument{

}

\makeatother

  \providecommand{\definitionname}{Definition}
  \providecommand{\propositionname}{Proposition}
  \providecommand{\remarkname}{Remark}
\providecommand{\theoremname}{Theorem}

\begin{document}

\lhead{Poisson Centralizer of the Trace}

\rhead{ Szabolcs M\'esz\'aros}

\title{Poisson Centralizer of the Trace}

\author{Szabolcs M\'esz\'aros}
\begin{abstract}
The Poisson centralizer of the trace element $\sum_{i}x_{i,i}$ is
determined in the coordinate ring of $SL_{n}$ endowed with the Poisson
structure obtained as the semiclassical limit of its quantized coordinate
ring. It turns out that this maximal Poisson-commutative subalgebra
coincides with the subalgebra of invariants with respect to the adjoint
action.
\thanks{2010\emph{ Mathematics Subject Classification.} 16T20, 17B63 (primary),
16W70, 20G42 (secondary).}
\end{abstract}

\maketitle

\section{Introduction}

The semiclassical limit Poisson structure on $\mathcal{O}(SL_{n})$
received considerable attention recently because of the connection
between the primitive ideals of the quantized coordinate ring $\mathcal{O}_{q}(SL_{n})$
and the symplectic leaves of the Poisson manifold $SL_{n}$ (see for
example \cite{HL2},\cite{G},\cite{Y}). In this paper, we present
another relation between $\mathcal{O}(SL_{n})$ endowed with the semiclassical
limit Poisson structure and $\mathcal{O}_{q}(SL_{n})$. 

In \cite{M} it was shown that if $q\in\mathbb{C}^{\times}$ is not
a root of unity then the centralizer of the trace element $\overline{\sigma}_{1}=\sum_{i}x_{i,i}$
in $\mathcal{O}_{q}(SL_{n})$ (resp. in $\mathcal{O}_{q}(M_{n})$
and $\mathcal{O}_{q}(GL_{n})$) is a maximal commutative subalgebra,
generated by certain sums of principal quantum minors. By Theorem
2.4 and 5.1 in \cite{DL2}, this subalgebra coincides with the subalgebra
of cocommutative elements in $\mathcal{O}_{q}(SL_{n})$ and also with
the subalgebra of invariants of the adjoint coaction. (This result
is generalized in \cite{AZ} for arbitrary characteristic and $q$
being a root of unity.) 

On the Poisson algebra side, the corresponding Poisson-subalgebra
of $\mathcal{O}(SL_{n})$ is generated by the coefficients of the
characteristic polynomial $\overline{c}_{1},\dots,\overline{c}_{n-1}$.
We prove the following:
\begin{thm}
For $n\geq1$ the subalgebra $\mathbb{C}[\overline{c}_{1},\dots,\overline{c}_{n-1}]$
(resp. $\mathbb{C}[c_{1},\dots,c_{n}]$ and $\mathbb{C}[c_{1},\dots,c_{n},c_{n}^{-1}]$)
is maximal Poisson-commutative in $\mathcal{O}(SL_{n})$ (resp. $\mathcal{O}(M_{n})$
and $\mathcal{O}(GL_{n})$) with respect to the semiclassical limit
Poisson structure.\label{thm:The-subalgebra}
\end{thm}
It is easy to deduce from \cite{DL1} or \cite{DL2} that $\{c_{i},c_{j}\}=0$
$(1\leq i,j\leq n)$ in $\mathcal{O}(M_{n})$ (see Proposition \ref{prop:commutation}
below). Therefore, Theorem \ref{thm:The-subalgebra} is a direct consequence
of the following statement:
\begin{thm}
For $n\geq1$ the Poisson-centralizer of $\overline{c}_{1}$ in $\mathcal{O}(SL_{n})$
(resp. $c_{1}\in\mathcal{O}(M_{n})$ and $\mathcal{O}(GL_{n})$) equipped
with the semiclassical limit Poisson bracket is generated as a subalgebra
by

\begin{itemize}
\item $\overline{c}_{1},\dots\overline{c}_{n-1}$ in the case of $\mathcal{O}(SL_{n})$,
\item $c_{1},\dots,c_{n}$ in the case of $\mathcal{O}(M_{n})$, and
\item $c_{1},\dots,c_{n},c_{n}^{-1}$ in the case of $\mathcal{O}(GL_{n})$.
\label{thm:The-centralizer-of-the-trace}
\end{itemize}
\end{thm}
The proof is based on modifying the Poisson bracket of the algebras
that makes an induction possible. A similar idea is used in the proof
of the analogous result in the quantum setup (see \cite{M}).

It is well known that the coefficient functions $c_{1},\dots,c_{n}\in\mathcal{O}(M_{n})$
of the characteristic polynomial generate the subalgebra $\mathcal{O}(M_{n})^{GL_{n}}$
of $GL_{n}$-invariants with respect to the adjoint action. This implies
that the subalgebra coincides with the Poisson center of the coordinate
ring $\mathcal{O}(M_{n})$ endowed with the Kirillov-Kostant-Souriau
(KKS) Poisson bracket. Hence, Theorem \ref{thm:The-subalgebra} for
$\mathcal{O}(M_{n})$ can be interpreted as an interesting interplay
between the KKS and the semiclassical limit Poisson structure. Namely,
while the subalgebra $\mathcal{O}(M_{n})^{GL_{n}}$ is contained in
every maximal Poisson-commutative subalgebra with respect to the former
Poisson bracket, it is contained in only one maximal Poisson-commutative
subalgebra (itself) with respect to the latter Poisson bracket.

A Poisson-commutative subalgebra is also called an involutive (or
Hamiltonian) system, while a maximal one is called a complete involutive
system (see Section \ref{sec:Prerequisites} or \cite{V}). Such a
system is integrable if the (Krull) dimension of the generated subalgebra
is sufficiently large. In our case, the subalgebra generated by the
elements $c_{1},\dots,c_{n-1}$ is not integrable, as its dimension
is $n-1$ (resp. $n$ for $GL_{n}$) instead of the required ${n+1 \choose 2}-1$
(resp. ${n+1 \choose 2}$ for $GL_{n}$), see Remark \ref{rem:The-rank}.

The article is organized as follows: First, we introduce the required
notions, and in Section \ref{sec:main} we prove that the three statements
in Theorem \ref{thm:The-centralizer-of-the-trace} are equivalent.
In Section \ref{prop:implications}, we prove Theorem \ref{thm:The-centralizer-of-the-trace}
for $n=2$ as a starting case of an induction presented in Section
\ref{sec:Proof} that completes the proof of the theorem. In the article,
every algebra is understood over the field $\mathbb{C}$.

\section{Preliminaries\label{sec:Prerequisites}}

\subsection{Poisson algebras}

First, we collect the basic notions about Poisson algebras we use
in the article. For further details about Poisson algebras, see \cite{V}. 

A commutative Poisson algebra $\big(A,\{.,.\}\big)$ is a unital commutative
associative algebra $A$ together with a bilinear operation $\{.,.\}:A\times A\to A$
called the Poisson bracket such that it is antisymmetric, satisfies
the Jacobi identity, and for any $a\in A$, $\{a,.\}:A\to A$ is a
derivation. For commutative Poisson algebras $A$ and $B$, the map
$\varphi:A\to B$ is a morphism of Poisson algebras if it is both
an algebra homomorphism and a Lie-homomorphism. 

There is a natural notion of Poisson subalgebra (i.e. a subalgebra
that is also a Lie-subalgebra), Poisson ideal (i.e an ideal that is
also a Lie-ideal) and quotient Poisson algebra (as the quotient Lie-algebra
inherits the bracket). The Poisson centralizer $C(a)$ of an element
$a\in A$ is defined as $\{b\in A\ |\ \{a,b\}=0\}$. Clearly, it is
a Poisson subalgebra. Analogously, $a\in A$ is called Poisson-central
if $C(a)=A$. One says that a subalgebra $C\leq A$ is Poisson-commutative
(or involutive) if $\{c,d\}=0$ for all $c,d\in C$ and it is maximal
Poisson-commutative (or maximal involutive) if there is no Poisson-commutative
subalgebra in $A$ that strictly contains $C$.

The Poisson center (or Casimir subalgebra) of $A$ is $Z(A):=\{a\in A\ |\ C(a)=A\}$.
Let $A$ be a reduced, finitely generated commutative Poisson algebra.
The rank $\mathrm{Rk}\{.,.\}$ of the Poisson structure $\{.,.\}$
is defined by the rank of the matrix $\big(\{g_{i},g_{j}\})_{i,j}\in A^{N\times N}$
for a generating system $g_{1},\dots,g_{N}\in A$. (One can prove
that it is independent of the chosen generating system.) A maximal
Poisson-commutative subalgebra $C$ is called integrable if
\[
\dim C=\dim A-\frac{1}{2}\mathrm{Rk}\{.,.\}
\]
The inequality $\leq$ holds for any Poisson-commutative subalgebra
(Proposition II.3.4 in \cite{V}), hence integrability is a maximality
condition on the size of $C$ that does not necessarily hold for every
maximal involutive system. 

\subsection{Filtered Poisson algebras}
\begin{defn}
A filtered Poisson algebra is a Poisson algebra together with an ascending
chain of subspaces $\{\mathcal{F}^{d}\}_{d\in\mathbb{N}}$ in $A$
such that 

\begin{itemize}
\item $A=\cup_{d\in\mathbb{N}}\mathcal{F}^{d}$,
\item $\mathcal{F}^{d}\cdot\mathcal{F}^{e}\subseteq\mathcal{F}^{d+e}$ for
all $d,e\in\mathbb{N}$, and 
\item $\{\mathcal{F}^{d},\mathcal{F}^{e}\}\subseteq\mathcal{F}^{d+e}$ for
all $d,e\in\mathbb{N}$.
\end{itemize}
Together with the filtration preserving morphisms of Poisson algebras,
they form a category. 
\end{defn}
For a filtered Poisson algebra $A$, we may define its associated
graded Poisson algebra $\mathrm{gr}A$ as
\[
\mathrm{gr}(A):=\bigoplus_{d\in\mathbb{N}}\mathcal{F}^{d}/\mathcal{F}^{d-1}
\]
where we used the simplifying notation $\mathcal{F}^{-1}=\{0\}$.
The multiplication of $\mathrm{gr}(A)$ is defined the usual way:
\[
\mathcal{F}^{d}/\mathcal{F}^{d-1}\times\mathcal{F}^{e}/\mathcal{F}^{e-1}\to\mathcal{F}^{d+e}/\mathcal{F}^{d+e-1}
\]
\[
\big(x+\mathcal{F}^{d-1},y+\mathcal{F}^{e-1}\big)\mapsto xy+\mathcal{F}^{d+e-1}
\]
Analogously, the Poisson structure of $\mathrm{gr}(A)$ is defined
by $\big(x+\mathcal{F}^{d-1},y+\mathcal{F}^{e-1}\big)\mapsto\{x,y\}+\mathcal{F}^{d+e-1}$.
One can check that this way $\mathrm{gr}(A)$ is a Poisson algebra. 

Let $(S,+)$ be an abelian monoid. (We will only use this definition
for $S=\mathbb{N}$ and $S=\mathbb{Z}/n\mathbb{Z}$ for some $n\in\mathbb{N}$.)
An $S$-graded Poisson algebra $R$ is a Poisson algebra together
with a fixed grading
\[
R=\oplus_{d\in S}R_{d}
\]
such that $R$ is both a graded algebra (i.e. $R_{d}\cdot R_{e}\subseteq R_{d+e}$
for all $d,e\in S$) and a graded Lie algebra (i.e. $\{R_{d},R_{e}\}\subseteq R_{d+e}$
for all $d,e\in S$) with respect to the given grading.

The above construction $A\mapsto\mathrm{gr}(A)$ yields an $\mathbb{N}$-graded
Poisson algebra. In fact, $\mathrm{gr}(.)$ can be turned into a functor:
for a morphism of filtered Poisson algebras $f:\big(A,\{\mathcal{F}^{d}\}_{d\in\mathbb{N}}\big)\to\big(B,\{\mathcal{G}^{d}\}_{d\in\mathbb{N}}\big)$
we define
\[
\mathrm{gr}(f):\mathrm{gr}(A)\to\mathrm{gr}(B)\qquad\big(x_{d}+\mathcal{F}^{d-1}\big)_{d\in\mathbb{N}}\mapsto\big(f(x_{d})+\mathcal{G}^{d-1}\big)_{d\in\mathbb{N}}
\]
One can check that it is indeed well defined and preserves composition. 
\begin{rem}
Given an $\mathbb{N}$-graded Poisson algebra $R=\oplus_{d\in\mathbb{N}}R_{d}$,
one has a natural way to associate a filtered Poisson algebra to it.
Namely, let $\mathcal{F}^{d}:=\oplus_{k\leq d}R_{k}$. In this case,
the associated graded Poisson algebra $\mathrm{gr}R$ of $\big(R,\{\mathcal{F}^{d}\}_{d\in\mathbb{N}}\big)$
is isomorphic to $R$.
\end{rem}

\subsection{The Kirillov-Kostant-Souriau bracket\label{subsec:The-Kirillov-Kostant-Souriau-bracket}}

A classical example of a Poisson algebra is given by the Kirillov-Kostant-Souriau
(KKS) bracket on $\mathcal{O}(\mathfrak{g}^{*})$, the coordinate
ring of the dual of a finite-dimensional (real or complex) Lie algebra
$\big(\mathfrak{g},[.,.]\big)$ (see \cite{ChP} Example 1.1.3, or
\cite{W} Section 3). 

It is defined as follows: a function $f\in\mathcal{O}(\mathfrak{g}^{*})$
at a point $v\in\mathfrak{g}^{*}$ has a differential $\mathrm{d}f_{v}\in T_{v}^{*}\mathfrak{g}^{*}$
where we can canonically identify the spaces $T_{v}^{*}\mathfrak{g}^{*}\cong T_{0}^{*}\mathfrak{g}^{*}\cong\mathfrak{g}^{**}\cong\mathfrak{g}$.
Hence, we may define the Poisson bracket on $\mathcal{O}(\mathfrak{g}^{*})$
as 
\[
\{f,g\}(v):=[\mathrm{d}f_{v},\mathrm{d}g_{v}](v)
\]
for all $f,g\in\mathcal{O}(\mathfrak{g}^{*})$ and $v\in\mathfrak{g}^{*}$.
It is clear that it is a Lie-bracket but it can be checked that the
Leibniz-identity is also satisfied. For $\mathfrak{g}=\mathfrak{gl}_{n}$,
it gives a Poisson bracket on $\mathcal{O}(M_{n})$.

Alternatively, one can define this Poisson structure via semiclassical
limits.

\subsection{Semiclassical limits\label{subsec:Semi-classical-limits}}

Let $A=\cup_{d\in\mathbb{Z}}\mathcal{A}^{d}$ be a $\mathbb{Z}$-filtered
algebra such that its associated graded algebra $\mathrm{gr}(A):=\oplus_{d\in\mathbb{Z}}\mathcal{A}^{d}/\mathcal{A}^{d-1}$
is commutative. The Rees ring of $A$ is defined as
\[
\mathrm{Rees}(A):=\bigoplus_{d\in\mathbb{Z}}\mathcal{A}^{d}h^{d}\subseteq A[h,h^{-1}]
\]
Using the obvious multiplication, it is a $\mathbb{Z}$-graded algebra.
The semiclassical limit of $A$ is the Poisson algebra $\mathrm{Rees}(A)/h\mathrm{Rees}(A)$
together with the bracket
\[
\{a+h\mathcal{A}^{m},b+h\mathcal{A}^{n}\}:=\frac{1}{h}[a,b]+\mathcal{A}^{n+m-2}\in\mathcal{A}^{n+m-1}/\mathcal{A}^{n+m-2}
\]
for all homogeneous elements $a+h\mathcal{A}^{m}\in\mathcal{A}^{m+1}/h\mathcal{A}^{m}$,
$b+h\mathcal{A}^{n}\in\mathcal{A}^{n+1}/h\mathcal{A}^{n}$. The definition
is valid as the underlying algebra of $\mathrm{Rees}(A)/h\mathrm{Rees}(A)$
is $\mathrm{gr}(A)$ that is assumed to be commutative, hence $[a,b]\in h\mathcal{A}^{m+n-1}$.

The Poisson algebra $\mathcal{O}(\mathfrak{g}^{*})$ with the KKS
bracket can be obtained as the semiclassical limit of $U\mathfrak{g}$,
see \cite{G}, Example 2.6.

\subsection{Quantized coordinate rings\label{subsec:Quantized-coordinate-rings}}

Assume that $n\in\mathbb{N}^{+}$ and define $\mathcal{O}_{t}(M_{n})$
as the unital $\mathbb{C}$-algebra generated by the $n^{2}$ generators
$x_{i,j}$ for $1\leq i,j\leq n$ over $\mathbb{C}[t,t^{-1}]$ that
are subject to the following relations:
\[
x_{i,j}x_{k,l}=\begin{cases}
x_{k,l}x_{i,j}+(t-t^{-1})x_{i,l}x_{k,j} & \textrm{if }i<k\textrm{ and }j<l\\
tx_{k,l}x_{i,j} & \textrm{if }(i=k\textrm{ and }j<l)\textrm{ or }(j=l\textrm{ and }i<k)\\
x_{k,l}x_{i,j} & \textrm{if }(i>k\textrm{ and }j<l)\textrm{ or }(j>l\textrm{ and }i<k)
\end{cases}
\]
for all $1\leq i,j,k,l\leq n$. It turns out to be a finitely generated
$\mathbb{C}[t,t^{-1}]$-algebra that is a Noetherian domain. (For
a detailed exposition, see \cite{BG}.) Furthermore, it can be endowed
with a coalgebra structure by setting $\varepsilon(x_{i,j})=\delta_{i,j}$
and $\Delta(x_{i,j})=\sum_{k=1}^{n}x_{i,k}\otimes x_{k,j}$. It turns
$\mathcal{O}_{t}(M_{n})$ into a bialgebra. 

For $q\in\mathbb{C}^{\times}$, the quantized coordinate ring of $n\times n$
matrices with parameter $q$ is defined as the $\mathbb{C}$-algebra
\[
\mathcal{O}_{q}(M_{n}):=\mathcal{O}_{t}(M_{n})/(t-q)
\]
In this article, we only deal with the case when $q$ is not a root
of unity, then the algebra is called the generic quantized coordinate
ring of $M_{n}$.

Similarly, one can define the non-commutative deformations of the
coordinate rings of $GL_{n}$ and $SL_{n}$ using the quantum determinant

\[
\mathrm{det}_{q}:=\sum_{s\in S_{n}}(-q)^{\ell(s)}x_{1,s(1)}x_{2,s(2)}\dots x_{n,s(n)}
\]
where $\ell(\sigma)$ stands for the length of $\sigma$ in the Coxeter
group $S_{n}$. Then \textendash{} analogously to the classical case
\textendash{} one defines 
\[
\mathcal{O}_{q}(SL_{n}):=\mathcal{O}_{q}(M_{n})/(\mathrm{det}_{q}-1)\qquad\mathcal{O}_{q}(GL_{n}):=\mathcal{O}_{q}(M_{n})\big[\mathrm{det}_{q}^{-1}\big]
\]
by localizing at the central element $\det_{q}$.

\subsection{Semiclassical limits of quantized coordinate rings}

The semiclassical limits of $\mathcal{O}_{q}(SL_{n})$ can be obtained
via the slight modification of process of Section \ref{subsec:Semi-classical-limits}
(see \cite{G}, Example 2.2). The algebra $R:=\mathcal{O}_{t}(M_{n})$
can be endowed with a $\mathbb{Z}$-filtration by defining $\mathcal{F}^{n}$
to be the span of monomials that are the product of at most $n$ variables.
However, instead of defining a Poisson structure on $\mathrm{Rees}(R)/h\mathrm{Rees}(R)$
with respect to this filtration, consider the algebra $R/(t-1)R$
that is isomorphic to $\mathcal{O}(M_{n})$ as an algebra. The semiclassical
limit Poisson bracket is defined as
\[
\{\bar{a},\bar{b}\}:=\frac{1}{t-1}(ab-ba)+(t-1)R\in R/(t-1)R
\]
for any two representing elements $a,b\in R$ for $\bar{a},\bar{b}\in R/(t-1)R$.
One can check that it is a well-defined Poisson bracket.

This Poisson structure of $\mathcal{O}(M_{n})$ can be given explicitly
by the following relations:
\[
\{x_{i,j},x_{k,l}\}=\begin{cases}
2x_{i,l}x_{k,j} & \textrm{if }i<k\textrm{ and }j<l\\
x_{i,j}x_{k,l} & \textrm{if }(i=k\textrm{ and }j<l)\textrm{ or }(j=l\textrm{ and }i<k)\\
0 & \textrm{otherwise}
\end{cases}
\]
extended according to the Leibniz-rule (see \cite{G}). It is a quadratic
Poisson structure in the sense of \cite{V}, Definition II.2.6. The
semiclassical limit for $GL_{n}$ and $SL_{n}$ is defined analogously
using $\mathcal{O}_{q}(GL_{n})$ and $\mathcal{O}_{q}(SL_{n})$ or
by localization (resp. by taking quotient) at the Poisson central
element $\mathrm{det}$ (resp. $\mathrm{det}-1$) in $\mathcal{O}(M_{n})$.

\subsection{Coefficients of the characteristic polynomial}

Consider the characteristic polynomial function $M_{n}\to\mathbb{C}[x]$,
$A\mapsto\mathrm{det}(A-xI)$. Let us define the elements $c_{0},c_{1},\dots,c_{n}\in\mathcal{O}(M_{n})$
as
\[
\mathrm{det}(A-xI)=\sum_{i=0}^{n}(-1)^{i}c_{i}x^{n-i}
\]
In particular, $c_{0}=1$, $c_{1}=\mathrm{tr}$ and $c_{n}=\mathrm{det}$.
Their images in $\mathcal{O}(SL_{n})\cong\mathcal{O}(M_{n})/(\mathrm{det}-1)$
are denoted by $\overline{c}_{1},\dots,\overline{c}_{n-1}$. If ambiguity
may arise, we will write $c_{i}(A)$ for the element corresponding
to $c_{i}$ for an algebra $A$ with a fixed isomorphism $A\cong\mathcal{O}(M_{k})$
for some $k$.

The coefficient functions $c_{1},\dots,c_{n}$ can also be expressed
via matrix minors as follows: For $I,J\subseteq\{1,\dots,n\}$, $I=(i_{1},\dots,i_{k})$
and $J=(j_{1},\dots,j_{k})$ define
\[
[I\,|\,J]:=\sum_{s\in S_{k}}\mathrm{sgn}(s)x_{i_{1},j_{s(1)}}\dots x_{i_{k},j_{s(k)}}
\]
i.e. it is the determinant of the subalgebra generated by $\{x_{i,j}\}_{i\in I,j\in J}$
that can be identified with $\mathcal{O}(M_{k})$. Then

\[
c_{i}=\sum_{|I|=i}[I\,|\,I]\in\mathcal{O}(M_{n})
\]
for all $1\leq i\leq n$. It is well.known that $c_{1},\dots,c_{n}$
generate the same subalgebra of $\mathcal{O}(M_{n})$ as the trace
functions $A\mapsto\mathrm{Tr}(A^{k})$, namely, the subalgebra $\mathcal{O}(M_{n})^{GL_{n}}$
of $GL_{n}$-invariants with respect to the adjoint action.

\section{Equivalence of the statements\label{sec:main}}

Consider $\mathcal{O}(M_{n})$ endowed with the semiclassical limits
Poisson bracket. As it is discussed in the Introduction, Theorem \ref{thm:The-subalgebra}
follows directly from Theorem \ref{thm:The-centralizer-of-the-trace}
and Proposition \ref{prop:commutation}.

The following proposition shows that it is enough to prove Theorem
\ref{thm:The-centralizer-of-the-trace} for the case of $\mathcal{O}(M_{n})$.
\begin{prop}
For any $n\in\mathbb{N}^{+}$ the following are equivalent:\label{prop:implications}

\begin{enumerate}
\item The Poisson-centralizer of $c_{1}\in\mathcal{O}(M_{n})$ is generated
by $c_{1},\dots,c_{n}$.
\item The Poisson-centralizer of $c_{1}\in\mathcal{O}(GL_{n})$ is generated
by $c_{1},\dots,c_{n},c_{n}^{-1}$.
\item The Poisson-centralizer of $\overline{c}_{1}\in\mathcal{O}(SL_{n})$
is generated by $\overline{c}_{1},\dots,\overline{c}_{n-1}$.
\end{enumerate}
\end{prop}

\begin{proof}
The first and second statements are equivalent as $\det$ is a Poisson-central
element, so we have $\{c_{1},h\cdot\mathrm{det}^{k}\}=\{c_{1},h\}\cdot\mathrm{det}^{k}$
for any $h\in\mathcal{O}(GL_{n})$ and $k\in\mathbb{Z}$. Hence, 
\[
\mathcal{O}(GL_{n})\supseteq C(c_{1})=\big(\mathcal{O}(M_{n})\cap C(c_{1})\big)[\mathrm{det}^{-1}]
\]
proving $1)\iff2)$.

$1)\iff3)$: First, assume $1)$ and let $\overline{h}\in\mathcal{O}(SL_{n})$
such that $\{\overline{c}_{1},\overline{h}\}=0$. Since $\mathcal{O}(SL_{n})$
is $\mathbb{Z}/n\mathbb{Z}$-graded (inherited from the $\mathbb{N}$-grading
of $\mathcal{O}(M_{n})$) and $\overline{c}_{1}$ is homogeneous with
respect to this grading, its Poisson-centralizer is generated by $\mathbb{Z}/n\mathbb{Z}$-homogeneous
elements, so we may assume that $\overline{h}$ is $\mathbb{Z}/n\mathbb{Z}$-homogeneous. 

Let $k=\deg(\overline{h})\in\mathbb{Z}/n\mathbb{Z}$. Let $h\in\mathcal{O}(M_{n})$
be a lift of $\overline{h}\in\mathcal{O}(SL_{n})$ and consider the 
$\mathbb{N}$-homogeneous decomposition $h=\sum_{j=0}^{d}h_{jn+k}$ of $h$, where $h_{jn+k}$
is homogeneous of degree $jn+k$ for all $j\in\mathbb{N}$. Define
\[
h':=\sum_{j=0}^{d}h_{jn+k}\mathrm{det}^{d-j}\in\mathcal{O}(M_{n})_{dn+k}
\]
that is a homogeneous element of degree $dn+k$ representing $\overline{h}\in\mathcal{O}(SL_{n})$
in $\mathcal{O}(M_{n})$. Then $\{c_{1},h'\}\in(\det-1)\cap\mathcal{O}(M_{n})_{dn+k+1}$
since $\{\overline{c}_{1},\overline{h'}\}=\{\overline{c}_{1},\overline{h}\}=0$,
$c_{1}$ is homogeneous of degree $1$ and the Poisson-structure is
graded. Clearly, $(\det-1)\cap\mathcal{O}(M_{n})_{dn+k+1}=0$ hence
$\{c_{1},h'\}=0$. Applying $1)$ gives $h'\in\mathbb{C}[c_{1},\dots,c_{n}]$
so $\overline{h}\in\mathbb{C}[\overline{c}_{1},\dots,\overline{c}_{n-1}]$
as we claimed.

Conversely, assume $3)$ and let $h\in\mathcal{O}(M_{n})$ such that
$\{c_{1},h\}=0$. Since $c_{1}$ is $\mathbb{N}$-homogeneous, we
may assume that $h$ is also $\mathbb{N}$-homogeneous and so the
image $\overline{h}\in\mathcal{O}(SL_{n})$ of $h$ is $\mathbb{Z}/n\mathbb{Z}$-homogeneous.
By the assumption, $\overline{h}=p(\overline{c}_{1},\dots,\overline{c}_{n-1})$
for some $p\in\mathbb{C}[t_{1},\dots,t_{n-1}]$. Endow $\mathbb{C}[t_{1},\dots,t_{n}]$
with the $\mathbb{N}$-grading $\mathrm{deg}(t_{i})=i$. As $\overline{h}$
is $\mathbb{Z}/n\mathbb{Z}$-homogeneous, we may choose $p\in\mathbb{C}[t_{1},\dots,t_{n-1}]$
so that its homogeneous components are all of degree $dn+\mathrm{deg}(\overline{h})\in\mathbb{N}$
with respect to the above grading for some $d\in\mathbb{N}$.

By $h-p(c_{1},\dots,c_{n-1})\in(\det-1)$ and the assumptions on degrees,
we may choose a polynomial $q\in\mathbb{C}[t_{1},\dots,t_{n}]$ that
is homogeneous with respect to the above grading and $q(t_{1},\dots,t_{n-1},1)=p$.
Let $h':=h\cdot\mathrm{det}^{r}$ where $r:=\frac{1}{n}(\deg q-\deg h)\in\mathbb{Z}$
so $\deg(h')=\deg(q)\in\mathbb{N}$. Then 
\[
h'-q(c_{1},\dots,c_{n})\in(\det-1)\cap\mathcal{O}(M_{n})_{\mathrm{deg}q}=0
\]
hence $h'\in\mathbb{C}[c_{1},\dots,c_{n}]$ and $h\in\mathbb{C}[c_{1},\dots,c_{n},c_{n}^{-1}]$.
This is enough as $\mathbb{C}[c_{1},\dots,c_{n},c_{n}^{-1}]\cap\mathcal{O}(M_{n})=\mathbb{C}[c_{1},\dots,c_{n}]$
by the definitions.
\end{proof}

\section{Case of $\mathcal{O}(SL_{2})$\label{sec:Case-of n=00003D2}}

In this section, we prove Theorem \ref{thm:The-centralizer-of-the-trace}
for $\mathcal{O}(SL_{2})$ that is the first step of the induction
in the proof of the general case.

We denote by $a,b,c,d$ the generators $\overline{x}_{1,1},\overline{x}_{1,2},\overline{x}_{2,1},\overline{x}_{2,2}\in\mathcal{O}(SL_{2})$
and $\mathrm{tr}:=\overline{c}_{1}=a+d$. 
\begin{prop}
The centralizer of $\mathrm{tr}\in\mathcal{O}(SL_{2})$ is $\mathbb{C}[\mathrm{tr}]$.\label{prop:n=00003D2}
\end{prop}
By $ad-bc=1$ we have a monomial basis of $\mathcal{O}(SL_{2})$ consisting
of
\[
a^{i}b^{k}c^{l},\ b^{k}c^{l}d^{j},\ b^{k}c^{l}\qquad(i,j\in\mathbb{N}^{+},\,k,l\in\mathbb{N})
\]
The Poisson bracket on the generators is the following:
\[
\{a,b\}=ab\qquad\{a,c\}=ac\qquad\{a,d\}=2bc
\]
\[
\{b,c\}=0\ \qquad\{b,d\}=bd\qquad\{c,d\}=cd
\]
The action of $\{\mathrm{tr},.\}$ on the basis elements can be written
as 
\[
\big\{(a+d),a^{i}b^{k}c^{l}\big\}=
\]
\[
=(k+l)a^{i+1}b^{k}c^{l}-2ia^{i-1}b^{k+1}c^{l+1}-(k+l)a^{i}b^{k}c^{l}d
\]
\[
=(k+l)a^{i+1}b^{k}c^{l}-(2i+k+l)a^{i-1}b^{k+1}c^{l+1}-(k+l)a^{i-1}b^{k}c^{l}
\]
By the same computation on $b^{k}c^{l}$ and $b^{k}c^{l}d^{j}$ one
obtains
\[
\big\{(a+d),b^{k}c^{l}\big\}=(k+l)ab^{k}c^{l}-(k+l)b^{k}c^{l}d
\]
\[
\big\{(a+d),b^{k}c^{l}d^{j}\big\}=(k+l+2j)b^{k+1}c^{l+1}d^{j-1}+(k+l)b^{k}c^{l}d^{j-1}-(k+l)b^{k}c^{l}d^{j+1}
\]
\smallskip{}
Hence, for a polynomial $p\in\mathbb{C}[t_{1},t_{2}]$ and $i\geq1$:
\begin{eqnarray}
\big\{(a+d),a^{i}p(b,c)\big\} & = & a^{i+1}\sum_{m}m\cdot p_{m}(b,c)\label{eq:commutator n=00003D2 a*p(b,c)}\\
 &  & -a^{i-1}\sum_{m}\big((2i+m)bc+m\big)p_{m}(b,c)\nonumber 
\end{eqnarray}
where $p_{m}$ is the $m$-th homogeneous component of $p$. The analogous
computations for $p(b,c)d^{j}$ ($j\geq1$) and $p(b,c)$ give 
\begin{eqnarray}
\big\{(a+d),p(b,c)d^{j}\big\} & = & -d^{j+1}\sum_{m}m\cdot p_{m}(b,c)\label{eq:commutator n=00003D2 p(b,c)*d}\\
 &  & +d^{j-1}\sum_{m}\big((m+2j)bc+m\big)p_{m}(b,c)\nonumber 
\end{eqnarray}
\begin{equation}
\big\{(a+d),p(b,c)\big\}=(a-d)\sum_{m}m\cdot p_{m}(b,c)\label{eq:commutator n=00003D2 p(b,c)}
\end{equation}

\begin{proof}[Proof of Proposition \ref{prop:n=00003D2}]
 Assume that $0\neq g\in C(\mathrm{tr})$ and write it as 
\[
g=\sum_{i=1}^{\alpha}a^{i}r_{i}+\sum_{j=1}^{\beta}s_{j}d^{j}+u
\]
where $r_{i}$, $s_{j}$ and $u$ are elements of $\mathbb{C}[b,c]$,
and $\alpha$ and $\beta$ are the highest powers of $a$ and $d$
appearing in the decomposition. 

We prove that $r_{\alpha}\in\mathbb{C}\cdot1$. If $\alpha=0$ then
$r_{\alpha}=u$ so the $a^{i}b^{k}c^{l}$ terms ($i>0$) in $\{a+d,g\}$
are the same as the $a^{i}b^{k}c^{l}$ terms in $\{a+d,u\}$ by Eq.
\ref{eq:commutator n=00003D2 a*p(b,c)}, \ref{eq:commutator n=00003D2 p(b,c)*d}
and \ref{eq:commutator n=00003D2 p(b,c)}. However, by \ref{eq:commutator n=00003D2 p(b,c)},
these terms are nonzero if $u\notin\mathbb{C}$ and that is a contradiction.
Assume that $\alpha\geq1$ and for a fixed $k\in\mathbb{N}$ define
the subspace 
\[
\mathcal{A}^{k}:=\sum_{l\leq k}a^{l}\mathbb{C}[b,c,d]\subseteq\mathcal{O}(SL_{2})
\]
By $\big\{\mathrm{tr},\mathcal{A}^{\alpha-1}\big\}\subseteq\mathcal{A}^{\alpha}$
we have 
\[
\mathcal{A}^{\alpha}=\{\mathrm{tr},g\}+\mathcal{A}^{\alpha}=\big\{\mathrm{tr},a^{\alpha}r_{\alpha}+\mathcal{A}^{\alpha-1}\big\}+\mathcal{A}^{\alpha}=
\]
\[
=a^{\alpha}\{\mathrm{tr},r_{\alpha}\}+\alpha a^{\alpha-1}bcr_{\alpha}+\mathcal{A}^{\alpha}=a^{\alpha}\{\mathrm{tr},r_{\alpha}\}+\mathcal{A}^{\alpha}
\]
By Eq. \ref{eq:commutator n=00003D2 p(b,c)} it is possible only if $\{\mathrm{tr},r_{\alpha}\}=0$
so $r_{\alpha}\in\mathbb{C}[b,c]\cap C(\mathrm{tr})=\mathbb{C}\cdot1$.

If $\alpha>0$ we may simplify $g$ by subtracting polynomials of
$\mathrm{tr}$ from it. Indeed, by $r_{\alpha}\in\mathbb{C}^{\times}$
we have $g-r_{\alpha}\mathrm{tr}^{\alpha}\in\mathcal{A}^{\alpha-1}\cap C(\mathrm{tr})$
so we can replace $g$ by $g-r_{\alpha}\mathrm{tr}^{\alpha}$. Hence,
we may assume that $\alpha=0$. Then, again, $r_{\alpha}=u\in\mathbb{C}\cdot1\subseteq C(\mathrm{tr})$
so we may also assume that $u=0$. 

If $g$ is nonzero after the simplification, we get a contradiction.
Indeed, let $p(b,c)d^{\gamma}$ be the summand of $g$ with the smallest
$\gamma\in\mathbb{N}$. By the above simplifications, $\gamma\geq1$.
Then the coefficient of $d^{\gamma-1}$ in $\{\mathrm{tr},g\}$ is
the same as the coefficient of $d^{\gamma-1}$ in 
\[
\{\mathrm{tr},p(b,c)d^{\gamma}\}=\{\mathrm{tr},p(b,c)\}d^{\gamma}+2\gamma bcp(b,c)d^{\gamma-1}
\]
so it is $2\gamma bcp(b,c)d^{\gamma-1}$ that is nonzero if $p(b,c)\neq0$
and $\gamma\geq1$. That is a contradiction.
\end{proof}

\section{Proof of the main result\label{sec:Proof}}

Let $n\geq2$ and let us denote $A_{n}:=\mathcal{O}(M_{n})$.
\begin{prop}
$\mathbb{C}[\sigma_{1},\dots,\sigma_{n}]\leq A_{n}$ is a Poisson-commutative
subalgebra. \label{prop:commutation}
\end{prop}
\begin{proof}
Consider the principal quantum minor sums 
\[
\sigma_{i}=\sum_{|I|=i}\sum_{s\in S_{i}}t^{-\ell(s)}x_{i_{1},i_{s(1)}}\dots x_{i_{t},i_{s(t)}}\in\mathcal{O}_{t}(M_{n})
\]
When $A_{n}$ is viewed as the semiclassical limit $R/(t-1)R$ where
$R=\mathcal{O}_{t}(M_{n})$ (see Subsection \ref{subsec:Quantized-coordinate-rings}),
one can see that $\sigma_{i}$ represents $c_{i}\in R/(t-1)R\cong\mathcal{O}(M_{n})$.
In \cite{DL1}, it is proved that $\sigma_{i}\sigma_{j}=\sigma_{j}\sigma_{i}$
in $\mathcal{O}_{q}(M_{n})$ if $q$ is not a root of unity, in particular,
if $q$ is transcendental. 

Since the algebra $\mathcal{O}_{q}(M_{n})$ is defined over $\mathbb{Z}[q,q^{-1}]$,
the elements $\sigma_{1},\dots,\sigma_{n}$ (that are defined over
$\mathbb{Z}[q,q^{-1}]$) commute in $\mathcal{O}_{q}\big(M_{n}(\mathbb{Z})\big)\leq\mathcal{O}_{q}\big(M_{n}(\mathbb{C})\big)$
as well. Hence, $\sigma_{1},\dots,\sigma_{n}$ also commute after
extension of scalars, i.e. in the ring $\mathcal{O}_{q}\big(M_{n}(\mathbb{Z})\big)\otimes_{\mathbb{Z}}\mathbb{C}\cong\mathcal{O}_{t}\big(M_{n}(\mathbb{C})\big)$.
Consequently, in $A_{n}\cong R/(t-1)R$ the subalgebra $\mathbb{C}[c_{1},\dots,c_{n}]$
is a Poisson-commutative subalgebra, by the definition of semiclassical
limit.
\end{proof}
By Proposition \ref{prop:commutation}, $\mathbb{C}[c_{1},\dots,c_{n}]$
is in the Poisson-centralizer $C(c_{1})$. To prove the converse,
Theorem \ref{thm:The-centralizer-of-the-trace}, we need further notations.
Consider the Poisson ideal
\[
I:=(x_{1,j},\,x_{i,1}\ |\ 2\leq i,j\leq n)\lhd A_{n}
\]
We will denote its quotient Poisson algebra by $B_{2,n}:=A_{n}/I$
and the natural surjection by $\varphi:A_{n}\to B_{2,n}$. Note that
$B_{2,n}\cong A_{n-1}[t]$ as Poisson algebras by $x_{i,j}+I\mapsto x_{i-1,j-1}$
($2\leq i,j\leq n$) and $x_{1,1}\mapsto t$ where the bracket of
$A_{n-1}[t]$ is the trivial extension of the bracket of $A_{n-1}$
by $\{t,a\}=0$ for all $a\in A_{n-1}[t]$. 

Furthermore, $D_{n}$ will stand for $\mathbb{C}[t_{1},\dots,t_{n}]$
endowed with the zero Poisson bracket. Define the map $\delta:B_{2,n}\to D_{n}$
as $x_{i,j}+I\mapsto\delta_{i,j}t_{i}$ that is morphism of Poisson
algebras by $\{x_{i,i},x_{j,j}\}\in I$. Note that $(\delta\circ\varphi)(c_{i})=s_{i}$,
the elementary symmetric polynomial in $t_{1},\dots,t_{n}$. In particular,
$\delta\circ\varphi$ restricted to $\mathbb{C}[c_{1},\dots,c_{n}]$
is an isomorphism onto the symmetric polynomials in $t_{1},\dots,t_{n}$
by the fundamental theorem of symmetric polynomials. In the proof
of Theorem \ref{thm:The-centralizer-of-the-trace} we verify the same
property for $C(\sigma_{1})$.

Although the algebras $A_{n}$, $B_{2,n}$ and $D_{n}$ are $\mathbb{N}$-graded
Poisson algebras (see Section \ref{sec:Prerequisites}) using the
total degree of $A_{n}$ and the induced gradings on the quotients,
we will instead consider them as filtered Poisson algebras where the
filtration is not the one that corresponds to this grading. For each
$d\in\mathbb{N}$, let us define 
\[
\mathcal{A}^{d}=\{a\in A_{n}\ |\ \deg_{x_{1,1}}(a)\leq d\}
\]
This is indeed a filtration on $A_{n}$. Note, that the grading $\deg_{x_{1,1}}$
is incompatible with the bracket by $\{x_{1,1},x_{2,2}\}=x_{1,2}x_{2,1}$.
The algebras $B_{2,n}$, $D_{n}$ and $C(c_{1})$ inherit a filtered
Poisson algebra structure as they are Poisson sub- and quotient algebras
of $A_{n}$ so we may take $\mathcal{B}^{d}:=\varphi(\mathcal{A}^{d})$,
$\mathcal{D}^{d}:=(\delta\circ\varphi)(\mathcal{A}^{d})$ and $\mathcal{C}^{d}=\mathcal{A}^{d}\cap C(c_{1})$.
This way, the natural surjections $\varphi$ and $\delta$ and the
embedding $C(c_{1})\hookrightarrow A_{n}$ are maps of filtered Poisson
algebras. 

In the proof of Theorem \ref{thm:The-centralizer-of-the-trace} we
use the associated graded Poisson algebras of $B_{2,n}$, $D_{n}$
and $C(c_{1})$ (see Section \ref{sec:Prerequisites}). First, we
describe the structure of these. The filtrations on $B_{2,n}$ and
$D_{n}$ are induced by the $x_{1,1}$- and $t_{1}$-degrees, hence
we have $\mathrm{gr}B_{2,n}\cong B_{2,n}$ and $\mathrm{gr}D_{n}\cong D_{n}$
as graded Poisson algebras (and $\mathrm{gr}\delta=\delta$), so we
identify them in the following. 

The underlying graded algebra of $\mathrm{gr}A_{n}$ is isomorphic
to $A_{n}$ using the $x_{1,1}$-degree but the Poisson bracket is
different: it is the same on the generators $x_{i,j}$ and $x_{k,l}$
for $(i,j)\neq(1,1)\neq(k,l)$ but 
\begin{eqnarray*}
\{x_{1,1},x_{i,j}\}_{\mathrm{gr}} & = & 0\qquad\quad(2\leq i,j\leq n)\\
\{x_{1,1},x_{1,j}\}_{\mathrm{gr}} & = & x_{1,1}x_{1,j}\quad(2\leq j\leq n)\\
\{x_{1,1},x_{i,1}\}_{\mathrm{gr}} & = & x_{1,1}x_{i,1}\quad(2\leq i\leq n)
\end{eqnarray*}
where $\{.,.\}_{\mathrm{gr}}$ stands for the Poisson bracket of $\mathrm{gr}A_{n}$.
Consequently, as maps we have $\mathrm{gr}\varphi=\varphi$, we still
have $\{c_{i},c_{j}\}_{\mathrm{gr}}=0$ for all $i,j$, and the underlying
algebra of $\mathrm{gr}C(c_{1})$ can be identified with $C(c_{1})$. 

Note, that $C(c_{1})$ is defined by the original Poisson structure
$\{.,.\}$ of $A_{n}$ and not by $\{.,.\}_{\mathrm{gr}}$, even if
it will be considered as a Poisson subalgebra of $\mathrm{gr}A_{n}$.
The reason of this slightly ambiguous notation is that we will also
introduce $C^{\mathrm{gr}}(x_{1,1})\subseteq\mathrm{gr}A_{n}$ as
the centralizer of $x_{1,1}$ with respect to $\{.,.\}_{\mathrm{gr}}$. 

Our associated graded setup can be summarized as follows:
\[
\xymatrix{C(c_{1})\ar@{}[r]|\subseteq & \mathrm{gr}A_{n}\ar@{->>}[r]^{\varphi} & B_{2,n}\ar@{->>}[r]^{\delta} & D_{n}}
\]

\begin{proof}[Proof of Theorem \ref{thm:The-centralizer-of-the-trace}]
 We prove the statement by induction on $n$. The statement is verified
for $\mathcal{O}(SL_{2})$ in Section \ref{sec:Case-of n=00003D2}
so, by Proposition \ref{prop:implications} the case $n=2$ is proved.
Assume that $n\geq3$. We shall prove that

\begin{itemize}
\item $(\delta\circ\varphi)|_{C(c_{1})}:C(c_{1})\to D_{n}$ is injective,
and
\item the image $(\delta\circ\varphi)\big(C(c_{1})\big)$ is in $D_{n}^{S_{n}}$.
\end{itemize}
These imply that the restriction of $\delta\circ\varphi$ to $C(c_{1})$
is an isomorphism onto $D_{n}^{S_{n}}$ since $C(c_{1})\ni c_{i}$
for $i=1,\dots,n$ (see Section \ref{sec:Prerequisites}) and $\delta\circ\varphi$
restricted to $\mathbb{C}[c_{1},\dots,c_{n}]$ is surjective onto
$D_{n}^{S_{n}}$. The statement of the theorem follows.

To prove that $\delta\circ\varphi$ is injective on $C(c_{1})$ it
is enough to prove that $\delta$ is injective on $C\big(\varphi(c_{1})\big)$
and that $\varphi$ is injective on $C(c_{1})$. Indeed, as $\varphi$
is a Poisson map we have $\varphi\big(C(c_{1})\big)\subseteq C\big(\varphi(c_{1})\big)$.

First, we prove $\delta$ is injective on $C\big(\varphi(c_{1})\big)$.
By $B_{2,n}\cong A_{n-1}[t]$ where $t$ is Poisson-central, we have
\[
B_{2,n}\supseteq C\big(\varphi(c_{1})\big)\cong C_{A_{n-1}[t]}\big(t+c_{1}(A_{n-1})\big)=C_{A_{n-1}}\big(c_{1}(A_{n-1})\big)[t]\subseteq A_{n-1}[t]
\]
By the induction hypothesis
\[
C_{A_{n-1}}\big(c_{1}(A_{n-1})\big)=\mathbb{C}\big[c_{1}(A_{n-1}),\dots,c_{n-1}(A_{n-1})\big]
\]
Therefore, $\delta$ restricted to $C\big(\varphi(c_{1})\big)$ is
an isomorphism onto $\mathbb{C}[s_{1},\dots,s_{n-1}][t_{1}]\subseteq D_{n}$
where $s_{i}$ is the symmetric polynomial in the variables $t_{2},\dots,t_{n}$.
In particular, $\delta$ is injective on $C\big(\varphi(c_{1})\big)$.

To verify the injectivity of $\varphi$ on $C(c_{1})$, define 
\[
C^{\mathrm{gr}}(x_{1,1}):=\{a\in\mathrm{gr}A_{n}\ |\ \{x_{1,1},a\}_{\mathrm{gr}}=0\}
\]
The subalgebra $C(c_{1})$ is contained in $C^{\mathrm{gr}}(x_{1,1})$
since for a homogeneous element $a$ of degree $d$, we have 
\[
\mathcal{A}^{d+1}/\mathcal{A}^{d}\ni\{x_{1,1},a\}_{\mathrm{gr}}+\mathcal{A}^{d}=\{x_{1,1}+\mathcal{A}^{0},a+\mathcal{A}^{d-1}\}+\mathcal{A}^{d}=\{c_{1},a\}+\mathcal{A}^{d}
\]
hence $\{c_{1},a\}=0$ implies $\{x_{1,1},a\}_{\mathrm{gr}}=0\in\mathrm{gr}A_{n}$.
Our setup can be visualized on the following diagram: 
\[
\xymatrix{\mathrm{gr}(A_{n})\ar@{->>}[r]^{\varphi} & B_{2,n}\ar@{->>}[r]^{\delta} & D_{n}\\
C^{\mathrm{gr}}(x_{1,1})\ar@{}[u]|\bigcup & C\big(\varphi(c_{1})\big)\ar@{}[u]|\bigcup\ar@{^{(}->}[ru]\\
C(c_{1})\ar@{}[u]|\bigcup\ar[ru]
}
\]
Now, it is enough to prove that $\varphi$ restricted to $C^{\mathrm{gr}}(x_{1,1})$
is injective.

We can give an explicit description of $C^{\mathrm{gr}}(x_{1,1})$
in the following form: 
\[
C^{\mathrm{gr}}(x_{1,1})=\mathbb{C}[x_{1,1},x_{i,j}\ |\ 2\leq i,j\leq n]\leq\mathrm{gr}A_{n}
\]
Indeed, 
\[
\{x_{1,1},x_{i,j}\}_{\mathrm{gr}}=\begin{cases}
x_{1,1}x_{i,j} & \textrm{if }j\neq i=1\textrm{ or }i\neq j=1\\
0 & \textrm{otherwise}
\end{cases}
\]
Therefore, the map $\mathrm{ad}_{\mathrm{gr}}x_{1,1}:\,a\mapsto\{x_{1,1},a\}_{\mathrm{gr}}$
acts on a monomial $m\in\mathrm{gr}A_{n}$ as $\{x_{1,1},m\}_{\mathrm{gr}}=c(m)\cdot x_{1,1}m$
where $c(m)$ is the sum of the exponents of the $x_{1,j}$'s and
$x_{i,1}$'s ($2\leq i,j\leq n$) in $m$. Hence, $\mathrm{ad}_{\mathrm{gr}}x_{1,1}$
maps the monomial basis of $\mathrm{gr}A_{n}$ injectively into itself.
In particular, 
\[
C^{\mathrm{gr}}(x_{1,1})=\mathrm{Ker}\big(\mathrm{ad}_{\mathrm{gr}}x_{1,1}\big)=\{a\in\mathrm{gr}A_{n}\ |\ c(m)=0\}\cong A_{n-1}[t]
\]
 using the isomorphism $x_{1,1}\mapsto t$ and $x_{i,j}\mapsto x_{i-1,j-1}$. 

The injectivity part of the theorem follows: $\varphi$ is injective
on $C^{\mathrm{gr}}(x_{1,1})$ (in fact it is an isomorphism onto
$B_{2,n}$), and $\varphi$ maps $C(c_{1})$ into $C\big(\varphi(c_{1})\big)$
on which $\delta$ is also injective.

To prove $(\delta\circ\varphi)\big(C(c_{1})\big)\subseteq D_{n}^{S_{n}}$,
first note that in the above we have proved that
\[
(\delta\circ\varphi)\big(C(c_{1})\big)\subseteq\delta\big(C\big(\varphi(c_{1})\big)\big)\subseteq D_{n}^{S_{n-1}}
\]
where $S_{n-1}$ acts on $D_{n}$ by permuting $t_{2},\dots,t_{n}$.
Consider the automorphism $\gamma$ of $A_{n}$ given by the reflection
to the off-diagonal: $\gamma(x_{i,j})=x_{n+1-i,n+1-j}$. It is not
a Poisson map but a Poisson antimap (using the terminology of \cite{ChP}),
i.e. $\gamma(\{a,b\})=-\{\gamma(a),\gamma(b)\}$. It maps $c_{1}$
into itself and consequently $C(c_{1})$ into itself. For the analogous
involution $\overline{\gamma}:D_{n}\to D_{n}$, $t_{i}\mapsto t_{n+1-i}$
($i=1,\dots,n$) we have $(\delta\circ\varphi)\circ\gamma=\overline{\gamma}\circ(\delta\circ\varphi)$.
Hence, 
\[
(\delta\circ\varphi)\big(C(c_{1})\big)=(\delta\circ\varphi\circ\gamma)\big(C(c_{1})\big)=(\overline{\gamma}\circ\delta\circ\varphi)\big(C(c_{1})\big)\subseteq\overline{\gamma}\big(D_{n}^{S_{n-1}}\big)
\]
proving the symmetry of $(\delta\circ\varphi)\big(C(c_{1})\big)$
in $t_{1},\dots,t_{n-1}$, so it is symmetric in all the variables
by $n\geq3$.
\end{proof}
\begin{rem}
In contrast with Theorem \ref{thm:The-subalgebra}, in the case of
the KKS Poisson structure, every Poisson-commutative subalgebra contains
the Poisson center $\mathbb{C}[c_{1},\dots,c_{n}]$, see \cite{W}.
For a maximal commutative subalgebra with respect to the KKS bracket,
see \cite{KW}.
\end{rem}
\begin{rem}
\label{rem:The-rank}We prove that $\mathbb{C}[\overline{c}_{1},\dots\overline{c}_{n-1}]$
is not an integrable complete involutive system (see Section \ref{sec:Prerequisites}).
First, observe that the rank of the semiclassical Poisson bracket
of $\mathcal{O}(SL_{n})$ is $n(n-1)$. 

Indeed, by Section \ref{sec:Prerequisites}, the rank is the maximal
dimension of the symplectic leaves in $SL_{n}$. The symplectic leaves
in $SL_{n}$ are classified in \cite{HL1}, Theorem A.2.1, based on
the work of Lu, Weinstein and Semenov-Tian-Shansky \cite{LW}, \cite{S}.
The dimension of a symplectic leaf is determined by an associated
element of $W\times W$ where $W=S_{n}$ is the Weyl group of $SL_{n}$.
According to Proposition A.2.2, if $(w_{+},w_{-})\in W\times W$ then
the dimension of the corresponding leaves is 
\begin{equation}
\ell(w_{+})+\ell(w_{-})+\min\{m\in\mathbb{N}\ |\ w_{+}w_{-}^{-1}=r_{1}\cdot\dots\cdot r_{m}\ |\ r_{i}\textrm{ is a transposition for all }i\}\label{eq:rank}
\end{equation}
where $\ell(.)$ is the length function of the Weyl group that \textendash{}
in the case of $SL_{n}$ \textendash{} is the number of inversions
in a permutation. By the definition of inversion using elementary
transpositions, the above quantity is bounded by
\[
\ell(w_{+})+\ell(w_{-})+\ell(w_{+}w_{-}^{-1})
\]
The maximum of the latter is $n(n-1)$ since $\ell(w_{+})={n \choose 2}-\ell(w_{+}t)$
where $t=(n\dots1)$ stands for the longest element of $S_{n}$. Therefore,
\[
\ell(w_{+})+\ell(w_{-})+\ell(w_{+}w_{-}^{-1})=n(n-1)-\ell(w_{+}t)-\ell(w_{-}t)+\ell\big((w_{+}t)(w_{-}t)^{-1}\big)\leq n(n-1)
\]
because $\ell(gh)\leq\ell(g)+\ell(h)=\ell(g)+\ell(h^{-1})$ for all
$g,h\in S_{n}$. This maximum is attained on $w_{+}=w_{-}=t$, even
for the original quantity in Equation \ref{eq:rank}. Hence, $\mathrm{Rk}\{.,.\}=n(n-1)$
for $SL_{n}$ and $\mathrm{Rk}\{.,.\}=n(n-1)+1$ for $M_{n}$ and
$GL_{n}$. However, a complete integrable system should have dimension
\[
\dim SL_{n}-\frac{1}{2}\mathrm{Rk}\{.,.\}=n^{2}-1-{n \choose 2}={n+1 \choose 2}-1
\]
So it does not equal to $\dim\mathbb{C}[\overline{c}_{1},\dots\overline{c}_{n-1}]=n-1$
if $n>1$. Similarly, the system is non-integrable for $M_{n}$ and
$GL_{n}$.
\end{rem}

\medskip{}

\noun{\small{}Department of Mathematics, Central European University, Budapest, 1051}{\small \par}

\emph{E-mail address}: \texttt{meszaros\_szabolcs@phd.ceu.edu}
\end{document}